\documentclass[11pt]{amsart}

\usepackage[all]{xy}
\usepackage{latexsym}
\usepackage{rotating}
\usepackage{amsfonts}
\usepackage{amssymb,amscd}
\usepackage{amsmath}
\usepackage{graphics, color}
\usepackage{hyperref}

\newtheorem{theorem}{Theorem}[section]
\newtheorem{lemma}[theorem]{Lemma}     
\newtheorem{corollary}[theorem]{Corollary}
\newtheorem{proposition}[theorem]{Proposition}
\newtheorem{remark}[theorem]{Remark}
\newtheorem{definition}[theorem]{Definition}

\newcommand{\fX}{\mathfrak{X}}

\newcommand{\C}{\mathbb{C}}
\newcommand{\hm}{\mathrm{Hom}}
\newcommand{\SL}{\mathrm{SL}}
\newcommand{\SU}{\mathrm{SU}}

\newcommand{\quot}{/\!\!/}

\newcommand{\R}{\mathbb{R}}

\newcommand{\xb}{\mathbf{x}}
\newcommand{\yb}{\mathbf{y}}

\newcommand{\ub}{\mathbf{u}}
\newcommand{\vb}{\mathbf{v}}
\newcommand{\wb}{\mathbf{w}}
\newcommand{\id}{\mathbf{I}}
\newcommand{\tr}{\rm{tr}}

\def\R{\mathbb R}

\def\C{\mathbb C}

\def\N{\mathbb N}

\def\z{{\bf z}}
\def\w{{\bf w}}

\def\W{\mathbb W}

\def \ch{{\bf H}_{\C}}

\def\tr{{\rm{tr}}}

\title[Invariants of pairs]{Invariants of pairs in $\SL(4, \C)$ and $\SU(3, 1)$}

\author[K. Gongopadhyay]{Krishnendu Gongopadhyay}
\address{Indian Institute of Science Education and Research (IISER) Mohali, Knowledge City, S.A.S. Nagar,
Sector 81, P. O. Manauli, Pin 140306, India}

\email{krishnendug@gmail.com, krishnendu@iisermohali.ac.in}

\author[S. Lawton]{Sean Lawton}

\address{Department of Mathematical Sciences, George Mason University,
4400 University Drive,
Fairfax, Virginia  22030, USA}

\email{slawton3@gmu.edu}



\begin{document}

\begin{abstract}
We describe a minimal global coordinate system of order 30 on the $\SL(4,\C)$-character variety of a rank 2 free group. Using symmetry within this system, we obtain a smaller collection of 22 coordinates subject to 5 further real relations that determine conjugation classes of generic pairs of matrices in $\SU(3,1)$.
\end{abstract}

\maketitle

\section{Introduction}

There is a long tradition of work whose goal is to classify pairs of elements in ${\rm SL}(2, \R)$ or ${\rm SL}(2, \C)$ 
in connection with Fuchsian and Kleinian groups. The study goes back to the work of Vogt \cite{vogt} and Fricke \cite{fricke} who proved that a non-elementary free two-generator subgroup is determined up to conjugation by the traces of the generators and 
their product. This result was incremental in the development of Teichm\"uller theory, and in particular, it was used to provide Fenchel-Nielsen coordinates on Teichm\"uller space. For an up to date exposition of this work see Goldman \cite{gold2}. 

In an attempt to define analogous Fenchel-Nielsen coordinates for complex hyperbolic
quasi-Fuchsian representations of surface groups, Parker and Platis \cite{pp} proved a generalization of the result of Fricke-Vogt for two-generator subgroups $\langle A, B \rangle$ in ${\rm SU}(2,1)$ with loxodromic generators. Parker and Platis followed an approach that uses traces of the generators and a point on the cross-ratio variety. In another approach, Will \cite{will1, will2} proved that a free two-generator Zariski dense subgroup of ${\rm SU}(2,1)$ is determined by  traces of the generators and the traces of three more compositions of the generators. For a survey of these results see Parker \cite{parker}. The starting point of Will's work is the results of Wen \cite{wen} and Lawton \cite{law} which provide a minimal 
generating set for the moduli space of ${\rm SL}(3, \C)$-representations of a rank 2 free group $F_2$. Let 
$\langle A, B \rangle$ denote the discrete, free subgroup of ${\rm SU}(2,1)$  generated by $A$ and $B$.  To give Fenchel-Nielsen type coordinates on the ${\rm SU}(2,1)$-character variety of $F_2$, one should obtain a set of conjugation invariant 
quantities that generically determine $\langle A, B \rangle$ up to conjugation.

As shown in \cite{law}, the space of ${\rm SL}(3, \C)$-representations of 
$F_2$ is a branched double cover of $\C^8$, and as shown in 
\cite{FL} this moduli space is homotopic to a topological 8-sphere. The nine 
minimal inverse symmetric trace coordinates described in \cite{law}, along with 
the fact that $\SU(2,1)$ matrices $A$ have their inverses conjugate to the complex conjugate-transpose of $A$, directly shows that the moduli space of  ${\rm SU}(2, 
1)$-representations of $F_2$ is determined by exactly 5 of the 9 coordinates 
from the complex case.  Since the full relation was determined in these 
symmetric terms in \cite{law}, further reduction to the sign of the 5-th trace 
coordinate, the trace of the commutator of $A$ and $B$, is possible (see 
\cite{parker}).  Similar results hold for ${\rm SU}(3)$-representations of 
$F_2$, see \cite[Section 6.5]{FL}.

Things become more complicated in ${\rm SU}(3,1)$. The approach of Will is
difficult to adapt in this setting. A reason is that the minimal generating set 
for the moduli space of ${\rm SL}(4, \C)$-representations of $F_2$ consists of 
30 complex trace-invariants as shown in Theorem \ref{min-gen-sl} below.  This 
follows from work of Drensky and Sadikova \cite{ds}, and Djokovich \cite{do}.  
The full set of relations between these trace-invariants is not known. 
However, using symmetry as done in \cite{law}, we 
obtain, similar to the cases of $\SU(2,1)$ and $\SU(3)$, a reduction from the 30 minimal 
trace coordinates described in Corollary \ref{min-gen-sl-sym} to a collection 
of 22 complex trace coordinates subject to 5 real relations.  This results in 39 real coordinates that appear to be unshortenable (see Definition \ref{unshort-def}).  We also show that the smallest possible collection of such real coordinates is 30.  These coordinates determine any $\SU(3,1)$ representation of $F_2$ having a closed conjugation orbit (polystable representation); including Zariski-dense pairs in ${\rm 
SU}(3,1)$. This is the content of Theorem \ref{su31-traces}.

Using geometric methods similar to the construction of classical Fenchel-Nielsen
coordinates and that of Parker-Platis \cite{pp}, a minimal 
set of 15 real invariants for generic pairs of {\it loxodromic} elements in 
${\rm SU}(3,1)$ was obtained in \cite{gp2}.  Let $\ch^3$ be three 
dimensional complex hyperbolic space where ${\rm SU}(3,1)$ acts by
isometries. A pair of loxodromic elements in ${\rm SU}(3,1)$ is 
called \emph{reducible} if it preserves a totally geodesic subspace of 
$\ch^3$.  The subgroup $\langle A, B \rangle$ is called 
reducible if the pair $(A, B)$ is reducible. Our final theorem
proves that reducible two generator subgroups of  ${\rm SU}(3,1)$ generated by loxodromic pairs $( A, B)$ are 
determined by at most 11 real coordinates (see Theorem \ref{red-lox-traces}).

\subsection*{Acknowledgements}
We acknowledge support from U.S. National Science Foundation grants: DMS 1107452, 1107263, 1107367 ``RNMS: GEometric structures And Representation varieties" (the GEAR Network).  Gongopadhyay acknowledges partial support from NBHM India, Grant NBHM/R.P. 7/2013/Fresh/992. Lawton was partially supported by grants from the Simons Foundation (Collaboration \#245642) and the U.S. National Science Foundation (DMS \#1309376), and acknowledges support by the National Science Foundation under Grant No. 0932078000 while he was in residence at the Mathematical Sciences Research Institute in Berkeley, California, during the Spring 2015 semester.  We thank Shiv Parsad for helpful conversations early in this project, and Lawton thanks Chris Manon for helpful references.  Lastly, we thank an anonymous referee whose suggestions helped improve the readability of the paper.

\section{Preliminaries}\label{prel} 

\subsection{Character Varieties}\label{charvar}

Let $\Gamma=\langle \gamma_1,...,\gamma_r\ |\ R\rangle$ be a finitely generated discrete group (with relations $R$) and $G$ a connected Lie group.  The set of homomorphisms $\hm(\Gamma, G)$ naturally sits inside the product space $G^r$ via the evaluation mapping $\rho\mapsto (\rho(\gamma_1),...,\rho(\gamma_r))$.  Therefore, $\hm(\Gamma, G)$ inherits the subspace topology.  Define $\hm(\Gamma,G)^*$ to be all $\rho$ in $\hm(\Gamma,G)$ such that the conjugation orbit of $\rho$ is closed.  Such points are called polystable.  The $G$-character variety of $\Gamma$ is then the conjugation orbit space $\fX(\Gamma, G):=\hm(\Gamma, G)^*/G$.  When $G$ is a complex reductive affine algebraic group, $\hm(\Gamma, G)$ is an affine variety (cut out of the product variety $G^r$ by the words in $R$).  Then by \cite[Theorem 2.1]{FlLa4}, $\fX(\Gamma, G)$ is homeomorphic to the geometric points (with the Euclidean topology on an affine variety) of the Geometric Invariant Theory (GIT) quotient $\hm(\Gamma,G)\quot G:=\mathrm{Spec}(\C[\hm(\Gamma, G)]^G)$, where $\C[\hm(\Gamma, G)]^G$ is the ring of $G$-invariant polynomials in the coordinate ring $\C[\hm(\Gamma,G)]$.  Moreover, by \cite[Proposition 3.4]{FLR} the GIT quotient with this topology is homotopic to the non-Hausdorff quotient space $\hm(\Gamma,G)/G$.

When $\Gamma$ is a free group $F_r$, as is our concern in this paper, the topology of these spaces has recently garnered significant attention.  In particular, many of the homotopy groups have been computed; see \cite{FLR} and 
reference therein.  In this case, the geometric structure is largely influenced by its algebraic structure, which itself owes much to a non-commutative algebra described in \cite{La0}.  We summarize here for the reader.

Let $F^{+}_r$ be the free non-commutative monoid generated by the formal symbols $\{x_1,...,x_r\}$, and let $M^+_r$ be the monoid  
generated by $\{\xb_1,\xb_2,...,\xb_r\},$ where $\xb_k=(x_{ij}^k)$ are matrices in $rn^2$ indeterminates, under matrix multiplication and with identity $\id$ the $n\times n$ identity matrix.  There is a surjection $F^{+}_r\to M^+_r$, defined by mapping $x_i\mapsto \xb_i$.  Let $\wb \in M^+_r$ be the image of $w\in F^{+}_r$ under this map.  Further, let $|\cdot|$ be the function that takes a cyclically reduced word in $F_r$ to its word length.  Then by \cite{P1,Ra}, we know the ring of invariants $\C[\mathfrak{gl}(n,\C)^r]^{\SL(n,\C)}$ is generated by $\{\tr(\wb)\ |\ w\in F^{+}_r, \ |w|\leq n^2\}.$

As mentioned above, $\C[\fX(F_r,\SL(n,\C))]\cong\C[\SL(n,\C)^r]^{\SL(n,\C)}$, and since the determinant is conjugation invariant $\C[\fX(F_r,\SL(n,\C))]$ is further isomorphic to \begin{equation}\label{specialize} (\C[\mathfrak{gl}(n,\C)^r]/\Delta)^{\SL(n,\C)}\cong\C[\mathfrak{gl}(n,\C)^r]^{\SL(n,\C)}/\Delta,\end{equation} where $\Delta$ is the ideal generated by the $r$ polynomials $\mathrm{det}(\xb_k)-1$.  Since the characteristic polynomial allows one to write the determinant as a polynomial in traces of words, we conclude that $\C[\fX(F_r,\SL(n,\C))]$ is generated by $\{\tr(\wb)\ |\ w\in F^{+}_r, \ |w|\leq n^2\}$ as well.

Let $M_r^*$ extend the monoid $M_r^+$ by $\{\xb^*_1,\xb^*_2,...,\xb^*_r\},$ where $\xb^*_k$ is the matrix whose $(i,j)^{\text{th}}$ entry is $(-1)^{i+j}$ times the determinant obtained by removing the $j^{\text{th}}$ row and $i^{\text{th}}$ column of $\xb_k$. Observe that $(\xb\yb)^*=\yb^*\xb^*$ for all $\xb,\yb \in M^+_r$, and $\xb\xb^*=\det(\xb)\id.$  Now let $M_r=M_r^*/N_r$ where
$N_r$ is the normal sub-monoid generated by $\{\det(\xb_k)\id\ |\ 1\leq k\leq r\}$.  Observe $\xb^*=\xb^{-1}$ in $M_r$, so $M_r$ is a group. Now let $\C M_r$ be the group algebra defined over $\C$ with respect to matrix addition 
and scalar multiplication in $M_r$.  Then the trace map $\tr:\C M_r\to \C[\fX(F_r, \SL(n,\C))]$ forms a bridge between the non-commutative algebra $\C M_r$ and the moduli space $\fX(F_r, \SL(n,\C))$, built in the language of free groups.

This relationship has been exploited to obtain many geometric results in the case $n=3$.  In particular, the relationship to this non-commutative algebra was used in \cite{law} to completely describe the coordinate ring of $\fX(F_2, \SL(3,\C))$.  As a consequence, $\fX(F_2, \SL(3,\C))$ was seen to be a branched double cover over $\C^8$ whose branching locus was exactly the transpose fixed representations.   This in turn led to applications to real projective structures on surfaces (see \cite{La0}), the study of the Poisson geometry of $\fX(F_2, \SL(3,\C))$ in \cite{La4}, and was used in \cite{FL} to show that $\fX(F_2, \SL(3,\C))$ is homotopic to the 8-sphere $S^8$.  

In Section \ref{fourbyfour}, we use the above general relationship to non-commutative algebra to likewise address the $n=4$ and $r=2$ case.

To that end, we will need results from \cite{do}.  To an $\N$-graded algebra $A=\oplus_{i\in \N} A_i$ one has its Hilbert-Poincar\'e series $h(A)=\sum_{i\geq 0}(\dim A_i)t^i$.  Since $\C[\mathfrak{gl}(n,\C)^r]^{\SL(n,\C)}$ is clearly graded, it makes sense to consider its Hilbert-Poincar\'e polynomial.  In \cite{Teranishi2}, $h(\C[\mathfrak{gl}(4,\C)^2]^{\SL(4,\C)})$ is computed and in \cite{Teranishi} a maximal algebraically independent set of parameters is described for the same case.  Using this and the general fact mentioned above that $\C[\mathfrak{gl}(n,\C)^r]^{\SL(n,\C)}$ is generated by traces of words in generic matrices, Djokovich  (see \cite[Theorem 4.2]{do}) begins with Teranishi's system of parameters and systematically augments these coordinates with new elements of the ring $\C[\mathfrak{gl}(4,\C)^2]^{\SL(4,\C)}$ while recalculating the Hilbert-Poincar\'e polynomial of the subalgebras obtained until he gets the right coefficients for $h(\C[\mathfrak{gl}(4,\C)^2]^{\SL(4,\C)})$.  The result is a minimal generating set for $\C[\mathfrak{gl}(4,\C)^2]^{\SL(4,\C)}$.

In the next section we will use this generating set to obtain a minimal generating set for $\C[\fX(F_2,\SL(4,\C))]$ using the following fact:

\begin{proposition}{\cite[Proposition 21]{law2}}\label{unshort}
If $\{\tr(\wb_1),...,\tr(\wb_N)\}$ union $\{\tr(\xb_1^n),...,\tr(\xb_r^n)\}$ is a minimal generating set for $\C[\mathfrak{gl}(n,\C)^r\quot \SL(n,\C)]$, then $\{\tr(\wb_1),...,\tr(\wb_N)\}$ corresponds to an unshortenable generating set for $\C[\SL(n,\C)^r\quot \SL(n,\C)]$.
\end{proposition}

The proof of the above proposition essentially follows by exploiting the filtration induced by the isomorphism labeled \eqref{specialize} above.  This allows one to show that any relation in $\C[\SL(n,\C)^r\quot \SL(n,\C)]$ among the generating set $\{\tr(\wb_1),...,\tr(\wb_N)\}$ would lift to one in $\C[\mathfrak{gl}(n,\C)^r\quot \SL(n,\C)]$, which does not exist by the minimality assumption.

\section{\texorpdfstring{$\SL(4,\C)$}{SL(4,C)} character variety of a rank 2 
free group.}\label{fourbyfour}

In this section we first prove a general result about minimal embeddings of $\SL(n,\C)$-character varieties of free groups. Next, we determine what this dimension is for a rank $2$ free group and $n=4$.  Doing so gives particularly symmetric coordinates on the character variety that can used to obtain a set of coordinates for the corresponding $\SU(3,1)$-character variety.

\begin{definition}\label{unshort-def}In a finitely generated ring $R$, given any set of generators $\mathcal{G}$, there is a smallest sub-collection $\mathcal{G}_0\subset \mathcal{G}$ that generates $R$.  Such a collection is said to be {\bf unshortenable}.  A {\bf minimal generating set} for $R$ is an unshortenable generating set of minimal cardinality among all unshortenable generating sets.\end{definition}

The following lemma is implied by \cite[Proposition 1.5.15]{BH}, but we include a proof to be self-contained.

\begin{lemma}\label{graded}
Let $R$ be a finitely generated graded ring over a field.  Then any unshortenable generating set for $R$ is minimal.
\end{lemma}
\begin{proof}
Let $R_+$ be the irrelevant ideal in $R$.
Any generating set for $R$ gives a spanning set for the finite dimensional vector space $R_+/(R_+)^2$, and conversely, by Nakayama's Lemma any basis for $R_+/(R_+)^2$ lifts to a generating set for $R_+$ and consequently $R$.  Since the dimension of $R_+/(R_+)^2$ does not depend on any generating set the result follows.\end{proof}

From Proposition \ref{unshort}, minimal generating sets for $\C[\mathfrak{gl}(n,\C)^r\quot \SL(n,\C)]$ give unshortenable generating sets for $\C[\SL(n,\C)^r\quot \SL(n,\C)]$. We now show that these unshortenable sets are in fact minimal.

\begin{proposition}\label{mintrmingen}
Suppose $\{\tr(\wb_1),...,\tr(\wb_N)\}$ is an unshortenable generating set for $\C[\SL(n,\C)^r\quot \SL(n,\C)]$.  Let $M$ be the minimal cardinality of any generating set for $\C[\SL(n,\C)^r\quot \SL(n,\C)]$.  Then $N=M$.
\end{proposition}

\begin{proof}
Sikora (see \cite{PS}) showed the $n=2$ case (and the $n=1$ case is trivial).  We now give a proof for arbitrary $n$.  For the remainder of this proof, let $\fX=\SL(n,\C)^r\quot \SL(n,\C)$.

For any $[\rho]\in \fX$, by definition of the Zariski tangent space,  $T_{[\rho]}\fX=(\mathfrak{m}/\mathfrak{m}^2)^*$ where $\mathfrak{m}$ is the maximal ideal in $\C[\fX]$ associated to $[\rho]$.  Fix a generating set $\mathcal{G}$ for $\C[\fX]$ and let $g$ be its cardinality.  By the Nullstellensatz the number of generators for $\mathfrak{m}$ with respect to a presentation of $\C[\fX]$ determined by $\mathcal{G}$ is no greater than $g$.  Therefore, since the dimension of the Zariski tangent space does not depend on any generating set, we have $\dim T_{[\rho]}\fX\leq M$ where $M$ is the minimal number of generators (among all generating sets) for $\C[\fX]$. 

On the other hand, by definition, $M\leq N$ where $N$ is the minimal number of trace coordinates needed to generate $\C[\fX]$.  Thus, it suffices to show there exists $[\rho_0]\in \fX$ so $\dim T_{[\rho_0]}\fX=N$.  Indeed, let $\rho_0$ be the trivial 
representation (all generators map to the identity).  Then the tangent space at the identity representation is $T_0\left(\mathfrak{sl}(n,\C)^{ r}\quot \SL(n,\C)\right)$ by \cite[Lemma 3.16]{FlLa3}.  

However, $\mathfrak{sl}(n,\C)^{ r}\quot \SL(n,\C)$ is exactly $\mathfrak{gl}(n,\C)^r\quot \SL(n,\C)$ with the $r$ trace coordinates $\tr(\xb_i),\ 1\leq i\leq r,$ specialized to $0$.  In particular, the minimal number of trace coordinates needed to generate $\C[\mathfrak{sl}(n,\C)^{ r}\quot \SL(n,\C)]$ is exactly $N$ since the minimal number of trace coordinates needed to generate $\C[\mathfrak{gl}(n,\C)^r\quot \SL(n,\C)]$ is $N +r$ by \cite[Proposition 21]{law2}.  

The minimal number of generators for 
$\C[\mathfrak{sl}(n,\C)^{ r}\quot \SL(n,\C)]$ is the same as smallest number 
of trace coordinates needed to generate it by by Lemma \ref{graded}, since $\C[\mathfrak{sl}(n,\C)^{ 
r}\quot \SL(n,\C)]$ is graded and trace coordinates are homogeneous.  Therefore, 
the minimum number of generators for the ring $\C[\mathfrak{sl}(n,\C)^{ r}\quot 
\SL(n,\C)]$ is $N$.

Finally, the maximal ideal at $0$ in $\C[\mathfrak{sl}(n,\C)^{ r}\quot \SL(n,\C)]$, denoted here $\mathfrak{m}_0$, is the irrelevant ideal (consisting of all non-constant homogeneous polynomials) and by Nakayama's Lemma the minimal number of generators of that ideal (which is $N$ by the previous paragraph and definition of the irrelevant ideal) is the same as the dimension of the cotangent space $\mathfrak{m}_0/\mathfrak{m}_0^2= T^*_{[\rho_0]}\fX$.  Thus, $\dim T_{[\rho_0]}\fX=N$, as required. \end{proof}

As a corollary, if $\{\tr(\wb_1),...,\tr(\wb_N)\}$ is a minimal number of trace generators for $\C[\fX(F_r,\SL(n,\C))]$, then the smallest affine space that the moduli space $\fX(F_r,\SL(n,\C))$ algebraically embeds into is $\mathbb{A}^N$.

Let $\mathcal{G}_1$ be the elements in Table \ref{sop} and $\mathcal{G}_2$ be the elements in Table \ref{msg}.
\begin{table}[!ht]
\begin{center}
\begin{tabular}{|c|c|}
\hline
Word Length &  Generator\\
\hline
$1$ &  $\tr(\xb),\ \tr(\yb)$\\
\hline
$2$ & $\tr(\xb^2),\ \tr(\xb\yb),\ \tr(\yb^2)$\\
\hline
$3$ & $\tr(\xb^3),\ \tr(\xb^2\yb),\ \tr(\xb\yb^2),\ \tr(\yb^3)$\\
\hline
$4$ & $\tr(\xb^4),\ \tr(\xb^3\yb),\ \tr(\xb^2\yb^2),\ \tr(\xb\yb^3),\ \tr(\yb^4),\ \tr(\xb\yb\xb\yb)$\\
\hline
$6$ &  $\tr((\xb^2\yb)^2),\ \tr((\yb^2\xb)^2)$\\
\hline
\end{tabular}
\caption{$\mathcal{G}_1$}\label{sop}
\end{center}
\end{table}

\begin{table}[!ht]
\begin{center}
\begin{tabular}{|c|c|}
\hline
Word Length &  Generator\\
\hline
$5$ &  $ \tr(\xb^3\yb^2),\ \tr(\yb^3\xb^2)$\\
\hline
$6$ & $ \tr(\xb^2\yb^2\xb\yb),\ \tr(\yb^2\xb^2\yb\xb)$\\
\hline
$7$ & $ \tr(\xb^3\yb^2\xb\yb),\ \tr(\yb^3\xb^2\yb\xb)$ \\
\hline
$8$ & $ \tr(\xb^3\yb^2\xb^2\yb),\ \tr(\yb^3\xb^2\yb^2\xb),\ \tr(\xb^3\yb^3\xb\yb),\ \tr(\yb^3\xb^3\yb\xb)$\\
\hline
$9$ & $\tr(\xb^3\yb\xb^2\yb\xb\yb),\ \tr(\xb^2\yb^2\xb\yb\xb^2\yb),$ \\
 & $\tr(\yb^2\xb^2\yb\xb\yb^2\xb),\ \tr(\yb^3\xb\yb^2\xb\yb\xb) $\\
\hline
$10$ & $\tr(\xb^3\yb^3\xb^2\yb^2) $\\
\hline
\end{tabular}
\caption{$\mathcal{G}_2$}\label{msg}
\end{center}
\end{table}
\newpage
Then \cite{do} (also compare \cite{ds}) shows that $\mathcal{G}_1\cup \mathcal{G}_2$ is a minimal system of 32 generators for $\C[\mathfrak{gl}(4,\C)^2\quot \SL(4,\C)]$, where $\mathcal{G}_1$ is a system of parameters (in particular, a maximal algebraically independent set).

\begin{theorem}\label{min-gen-sl}
$\mathcal{G}_1\cup \mathcal{G}_2 -\{\tr(\xb^4),\ \tr(\yb^4)\}$ is a minimal system of $30$ generators for $\C[\fX(F_2,\SL(4,\C))]$, where $\mathcal{G}_1 -\{\tr(\xb^4),\ \tr(\yb^4)\}$ is a maximal set of $15$ algebraically independent elements.
\end{theorem}

\begin{proof}
The fact that these form an unshortenable set of 30 trace generators follows from \cite{do} and  \cite[Proposition 21]{law2}.  The fact that this set is minimal follows from Proposition \ref{mintrmingen}.  The fact that $\mathcal{G}_1 -\{\tr(\xb^4),\ \tr(\yb^4)\}$ is algebraically independent follows by the same reasoning employed in \cite{La5}; and it is maximal since the Krull dimension of $\fX(F_2,\SL(4,\C))$ is 15.\end{proof}

\begin{lemma}\label{sublemma}
Let $\mathcal{G}$ generate $\C[\fX(F_r,\SL(4,\C))]$ and suppose $\tr(\ub\xb^{3}\vb)$ is in $\mathcal{G}$.  Then $\mathcal{G}\cup \{\tr(\ub\xb^{-1}\vb)\}-\{\tr(\ub\xb^{3}\vb)\}$ remains a generating set as long as $\tr(\xb),$  $\tr(\xb^2)$, $\tr(\xb^{-1})$, $\tr(\ub\vb),$ $\tr(\ub\xb\vb),$ and $\tr(\ub\xb^2\vb)$ are in the subring generated by $\mathcal{G}-\{\tr(\ub\xb^{3}\vb)\}.$
\end{lemma}

\begin{proof}
The characteristic polynomial for $\SL(4,\C)$ is: $$\xb^4-\tr(\xb)\xb^3+\left(\frac{\tr(\xb)^2-\tr(\xb^2)}{2}\right)\xb^2-\tr(\xb^{-1})\xb+\id=0.$$  Multiplying through on the left by a word $\ub$ and on the right by $\xb^{-1}\vb$ for a word $\vb$ gives: $$\ub\xb^3\vb-\tr(\xb)\ub\xb^2\vb+\left(\frac{\tr(\xb)^2-\tr(\xb^2)}{2}\right)\ub\xb\vb-\tr(\xb^{-1})\ub\vb+\ub\xb^{-1}\vb=0.$$  Therefore, taking traces of both sides of this latter equation, we have that $-\tr(\ub\xb^{-1}\vb)=$ $$\tr(\ub\xb^3\vb)-\tr(\xb)\tr(\ub\xb^2\vb)+\left(\frac{\tr(\xb)^2-\tr(\xb^2)}{2}\right)\tr(\ub\xb\vb)-\tr(\xb^{-1})\tr(\ub\vb),$$ which proves the lemma.

We note that $\tr(\xb^{-1})=\frac{1}{3}\left(\tr(\xb^3)+\frac{1}{2}\tr(\xb)^3-\frac{3}{2}\tr(\xb)\tr(\xb^2)\right)$ so that if $\mathcal{G}$ contains $\{\tr(\xb),\tr(\xb^2),\tr(\xb^3)\}$, then $\mathcal{G}\cup\{\tr(\xb^{-1})\}-\{\tr(\xb^3)\}$ remains a generating set.\end{proof}

Let $\tau$ be the involutive outer automorphism that permutes $\xb$ and $\yb$ and let $\iota$ be the involutive outer automorphism that sends $\xb\mapsto \xb^{-1}$ and $\yb\mapsto \yb^{-1}$.  Clearly $\tau$ and $\iota$ act on $\fX(F_2,\SL(4,\C))$ and its coordinate ring.  

Let $\mathcal{S}$ be the elements in the following table, where we weight the word length of inverse letters to be 3 (with respect to Lemma \ref{sublemma}):

\begin{table}[!ht]
\begin{center}
\begin{tabular}{|c|c|}
\hline
Word Length &  Generator\\
\hline
$1$ & $\tr(\xb)$\\
\hline
$2$ & $\tr(\xb^2),\ \tr(\xb\yb)$\\
\hline
$3$ & $\tr(\xb^{-1}),\ \tr(\xb\yb^{-2})$\\
\hline
$4$ & $\tr(\xb^{-1}\yb),\ \tr(\xb^2\yb^2),\ \tr(\xb\yb\xb\yb)$\\
\hline
$5$ &  $ \tr(\xb^{-1}\yb^2)$\\
\hline
$6$ &  $\tr((\xb^2\yb)^2), \tr(\xb^2\yb^2\xb\yb)$\\
\hline
$7$ & $ \tr(\xb^{-1}\yb^2\xb\yb)$ \\
\hline
$8$ & $ \tr(\xb^{-1}\yb^2\xb^2\yb),\ \tr(\xb^{-1}\yb^{-1}\xb\yb)$\\
\hline
$9$ & $\tr(\xb^{-1}\yb\xb^2\yb\xb\yb),\ \tr(\xb^2\yb^2\xb\yb\xb^2\yb)$ \\
\hline
\end{tabular}
\caption{$\mathcal{S}$}\label{sym-msg}
\end{center}
\end{table}

\begin{corollary}\label{min-gen-sl-sym}
$\mathcal{S}\cup\tau(\mathcal{S})\cup\{\tr(\xb^{-1}\yb^{-1}\xb^2\yb^2)\}$ is a minimal set of 30 generators for $\C[\fX(F_2,\SL(4,\C))]$.
\end{corollary}

\begin{proof}
We work from largest word length down since $\C[\fX(F_2,\SL(4,\C))]$ inherits a filtration from the multigrading (by degree, or equivalently by word length) on $\C[\mathfrak{gl}(n,\C)^r\quot \SL(n,\C)]$.  Denote by $\mathcal{G}$ the set $\mathcal{G}_1\cup \mathcal{G}_2-\{\tr(\xb^4),\tr(\yb^4)\}$, and let
$\mathcal{G}(\ell)$ be the elements in $\mathcal{G}$ of word length (equivalently, filtered degree) less than or equal to $\ell$.  So $\mathcal{G}=\mathcal{G}(10)$.

We start by showing we can replace the generator $\tr(\xb^{3}\yb^{3}\xb^2\yb^2)$ by the trace coordinate $\tr(\xb^{-1}\yb^{-1}\xb^2\yb^2)$ to obtain a new generating set that agrees with $\mathcal{G}$ on $\mathcal{G}(9)$.  We do this in two steps.  First, with respect to Lemma \ref{sublemma}, we can replace  $\tr(\xb^{3}\yb^{3}\xb^2\yb^2)$ by $\tr(\xb^{-1}\yb^{3}\xb^2\yb^2)$ if $\tr(\yb^{3}\xb^2\yb^2)$, $\tr(\xb\yb^{3}\xb^2\yb^2)$, and $\tr(\xb^2\yb^{3}\xb^2\yb^2)$ are in the subring generated without $\tr(\xb^{3}\yb^{3}\xb^2\yb^2)$ included.  However, $\tr(\yb^{3}\xb^2\yb^2)$, $\tr(\xb\yb^{3}\xb^2\yb^2)$, and $\tr(\xb^2\yb^{3}\xb^2\yb^2)$ are all of word length less than or equal to 9.  Therefore, there exists a polynomials $P_1$, $P_2$, and $P_3$ in the variables $\mathcal{G}(9)=\mathcal{G}-\{\tr(\xb^{3}\yb^{3}\xb^2\yb^2)\}$ so that $\tr(\yb^{3}\xb^2\yb^2)=P_1$, $\tr(\xb\yb^{3}\xb^2\yb^2)=P_2$, and $\tr(\xb^2\yb^{3}\xb^2\yb^2)=P_3$.  We conclude that we may substitute $\tr(\xb^{3}\yb^{3}\xb^2\yb^2)$ by $\tr(\xb^{-1}\yb^{3}\xb^2\yb^2)$ to obtain an equivalent generating set.

Likewise, we can replace $\tr(\xb^{-1}\yb^{3}\xb^2\yb^2)$ by $\tr(\xb^{-1}\yb^{-1}\xb^2\yb^2)$ if $\tr(\xb\yb^2)$, $\tr(\xb^{-1}\yb\xb^2\yb^2)$, and $\tr(\xb^{-1}\yb^2\xb^2\yb^2)$ are in the subring generated by $\mathcal{G}(9)$.  However, $\tr(\xb\yb^2)$ is one of these generators, and $\tr(\xb^{-1}\yb\xb^2\yb^2)$ (similarly $\tr(\xb^{-1}\yb^2\xb^2\yb^2)$), using Lemma \ref{sublemma} in reverse, is polynomially equivalent to $\tr(\xb^3\yb\xb^2\yb^2)$ (similarly $\tr(\xb^3\yb^2\xb^2\yb^2)$) and so is a polynomial in the variables $\mathcal{G}(9)$, as required.

Using the same reasoning, we next substitute the generator $\tr(\xb^3\yb\xb^2\yb\xb\yb)$ (respectively, its image under $\tau$) with $\tr(\xb^{-1}\yb\xb^2\yb\xb\yb)$ (respectively, its image under $\tau$) by working in $\mathcal{G}(8)$.  Working our way down in a like manner, we replace $\tr(\xb^3\yb^3\xb\yb)$ by $\tr(\xb^{-1}\yb^{-1}\xb\yb)$, $\tr(\xb^3\yb^2\xb^2\yb)$ by $\tr(\xb^{-1}\yb^2\xb^2\yb)$, and $\tr(\xb^3\yb^2\xb\yb)$ by $\tr(\xb^{-1}\yb^2\xb\yb)$.

The fact that we can replace $$\{\tr(\xb^3),\tr(\yb^3),\tr(\xb^3\yb),\tr(\xb\yb^3),\tr(\xb^3\yb^2),\tr(\yb^3\xb^2)\}$$ with $$\{\tr(\xb^{-1}),\tr(\yb^{-1}),\tr(\xb^{-1}\yb),\tr(\xb\yb^{-1}),\tr(\xb^{-1}\yb^2),\tr(\yb^{-1}\xb^2)\}$$ follows directly from Lemma \ref{sublemma} as stated.

It remains to prove that $\tr(\xb\yb^{2})$ (and its image under $\tau$) can be replaced with $\tr(\xb\yb^{-2})$ (respectively, its image under $\tau$), after the above substitutions have been made.  But the characteristic polynomial $$\yb^4-\tr(\yb)\yb^3+\left(\frac{\tr(\yb)^2-\tr(\yb^2)}{2}\right)\yb^2-\tr(\yb^{-1})\yb+\id=0$$ implies $\tr(\xb\yb^2)=$ $$\tr(\yb)\tr(\xb\yb)-\left(\frac{\tr(\yb)^2-\tr(\yb^2)}{2}\right)\tr(\xb)+\tr(\yb^{-1})\tr(\xb\yb^{-1})-\tr(\xb\yb^{-2})$$ by multiplying through by $\yb^{-2}\xb$ and then taking the trace of the result.  Consequently, we see the substitution is possible.  This completes the proof.\end{proof}

\begin{remark}
Using the relations in \cite{drensky} we can obtain non-trivial relations of degrees 12, 13, and 14 for $\fX(F_2,\SL(4,\C))$; note, that degree 12 is the minimal degree of any non-trivial relation in these coordinates.
\end{remark}

\begin{remark}
In \cite{Teranishi} a generating set for $\C[\mathfrak{gl}(4,\C)^2\quot \SL(4,\C)]$ was described, and likewise a generating set for $\C[\fX(F_2,\SL(4,\C))]$ is given by the general results of \cite{SikoraGen}.  Neither generating set appears to be minimal however, and both result in generating sets different from ours.
\end{remark}

\section{\texorpdfstring{$\SU(3,1)$}{SU(3,1)} character variety of a rank 2 
free group.}

Let $V=\C^{3,1}$ be the complex vector space $\C^4$ equipped with the Hermitian form of signature (3,1) given by $$\langle\z,\w\rangle=\w^{\ast}H\z=z_{1}\bar w_4+z_2\bar w_2+z_3\bar w_3
+z_4\bar w_1,$$
where $\ast$ denotes complex-conjugate transpose. The matrix of the Hermitian form is given by 
\begin{center}
$H=\left[ \begin{array}{cccc}
           0 & 0 & 0 & 1\\
           0 & 1 & 0 & 0\\
 0 & 0 & 1 & 0\\
 1 & 0 & 0 & 0
          \end{array}\right]$
\end{center}

Then ${\rm SU}(3,1)$ is the group of unit determinant matrices which preserve the Hermitian form $\langle \ ,\ \rangle$.  Each such matrix $A$ satisfies the relation $$A^{-1}=H^{-1}A^{\ast}H.$$

In this section, based on the work in the previous section, we reduce the coordinates needed to determine a polystable orbit of a $\SU(3,1)$ representation of $F_2$.  This gives a global coordinate system (and semi-algebraic embedding) of the corresponding character variety.

First, however, we observe that the smallest set of real invariants we can hope 
for in this case consists of 30 elements.
\begin{proposition}\label{complexminrealmin}
Let $K$ be a real affine algebraic group, and suppose its complexification $G$ is connected and reductive.  Then the minimal number of generators for $\R[\fX(F_r,K)]$ and $\C[\fX(F_r,G)]$ are equal.
\end{proposition}

\begin{proof}
This pretty much follows from the discussion at the beginning of Section 6 of \cite{CFLO}.

Obviously, $\C[\fX(F_r,G)]:=\C[G^r]^G\subset\C[G^r]^K$.  On the other hand, any invariant $f\in \C[G^r]^K$ satisfies $k\cdot f(x)-f(x)=0$ for all $x\in G^r$ and $k\in K$, but since $K$ is Zariski dense in $G$, the relation $k\cdot f(x)-f(x)=0$ extends to $g\cdot f(x)-f(x)=0$ for all $x\in G^r$ and $g\in G$.  Thus,  $\C[G^r]^G=\C[G^r]^K$.  Consequently, $\C[\fX(F_r,G)]=\R[K^r]^K\otimes_\R\C=\R[\fX(F_r, K)]\otimes_\R \C$.  In other words, any generating set for $\C[\fX(F_r,G)]$ determines a generating set for $\R[\fX(F_r, K)]$, and any generating set for $\R[\fX(F_r, K)]$ extends by scalars to a generating set for $\C[\fX(F_r,G)]$.  The result follows.\end{proof}

\begin{theorem}\label{su31-traces}
The following 22 traces determine any $($polystable$)$ pair $\langle A,B\rangle$ up to conjugation where $A,B\in \SU(3,1)$:   
\begin{center}
\begin{tabular}{|c|c|}
\hline
Word Length &  Generator\\
\hline
$1$ & $\tr(\xb)$, $\tr(\yb)$\\
\hline
$2$ & $\tr(\xb^2),$ $\tr(\xb\yb)$, $\tr(\yb^2)$, $\tr(\xb^{-1}\yb)$\\
\hline
$3$ & $\tr(\xb\yb^{2})$, $\tr(\yb\xb^{2})$\\
\hline
$4$ & $\tr(\xb^2\yb^2),$ $\tr(\xb\yb\xb\yb)$,  $\tr(\xb^{-1}\yb^{-1}\xb\yb)$\\
\hline
$5$ &  $ \tr(\xb^{-1}\yb^2\xb\yb)$, $ \tr(\yb^{-1}\xb^2\yb\xb)$\\
\hline
$6$ &  $\tr((\xb^2\yb)^2)$, $\tr((\yb^2\xb)^2)$, $\tr(\xb^2\yb^2\xb\yb)$, $\tr(\xb^{-1}\yb^{-1}\xb^2\yb^2)$ \\

 & $\tr(\yb^2\xb^2\yb\xb)$,  $ \tr(\xb^{-1}\yb^2\xb^2\yb),$ $ \tr(\yb^{-1}\xb^2\yb^2\xb)$\\
\hline
$7$ & $\tr(\xb^{-1}\yb\xb^2\yb\xb\yb)$,  $\tr(\yb^{-1}\xb\yb^2\xb\yb\xb)$ \\
\hline
\end{tabular}
\end{center}
\end{theorem}

\begin{proof}
Since the inverse of a matrix in $\SU(3,1)$ is conjugate by the $(3,1)$-form $H$ to the complex-conjugate transpose of that matrix, we see that for any word $\wb$ in letter from $\SU(3,1)$, that $\tr(\wb^{-1})=\overline{\tr(\wb)}$.  Therefore, given the $\iota$ related elements in the symmetric generating set $\mathcal{S}\cup \tau(\mathcal{S})\cup \{\tr(\xb^{-1}\yb^{-1}\xb^2\yb^2)\}$, we can eliminate the following six traces:  $\tr(\xb^{-1})$, $\tr(\yb^{-1})$, $\tr(\xb^{-1}\yb^{2})$, $\tr(\yb^{-1}\xb^{2})$, $\tr(\yb^{-1}\xb),$ and $\tr(\yb^{-1}\xb^{-1}\yb\xb)$.  

Therefore, it remains to show that $\tr(\yb^2\xb^2\yb\xb\yb^2\xb)=\tr((\yb^2\xb)^2\xb\yb\xb)$ and its $\tau$-image can be freely eliminated.  Indeed, we have earlier derived that $\tr(\vb^2\ub)=$ $$\tr(\vb)\tr(\ub\vb)-\left(\frac{\tr(\vb)^2-\tr(\vb^2)}{2}\right)\tr(\ub)+\tr(\vb^{-1})\tr(\ub\vb^{-1})-\tr(\ub\vb^{-2}),$$ and so letting $\vb=\yb^2\xb$ and $\ub=\xb\yb\xb$, we derive that: $\tr((\yb^2\xb)^2\xb\yb\xb)$
\begin{eqnarray*}
&=&\tr(\yb^2\xb)\tr(\yb^2\xb^2\yb\xb)-\left(\frac{\tr(\yb^2\xb)^2-\tr((\yb^2\xb)^2)}{2}\right)\tr(\xb^2\yb)+\\
& &\tr((\yb^2\xb)^{-1})\tr((\xb\yb\xb)(\yb^2\xb)^{-1})-\tr((\xb\yb\xb)(\yb^2\xb)^{-2})\\
&=&\tr(\yb^2\xb)\tr(\yb^2\xb^2\yb\xb)-\left(\frac{\tr(\yb^2\xb)^2-\tr((\yb^2\xb)^2)}{2}\right)\tr(\xb^2\yb)+\\
& &\overline{\tr(\yb^2\xb)}\tr(\xb\yb^{-1})-\overline{\tr(\xb^{-1}\yb^2\xb\yb)},
\end{eqnarray*}
which is clearly in terms of the other generators (or their complex conjugates).\end{proof}

\begin{lemma} \label{wmg} { \cite[Lemma 6.2.5]{gold}}. 
Let $T$ be a transformation in ${\rm U}(n,1)$. If $\lambda$ is an eigenvalue of $T$, then $\overline{\lambda}^{\,-1}$ is also an eigenvalue with the same multiplicity as that of $\lambda$.  \end{lemma}

The above lemma implies the following useful fact.

\begin{lemma}\label{realrel}
Let $A$ be an element in ${\rm SU}(3,1)$. Then $$\sigma(A)= \frac{1}{2}(\tr^2(A)-\tr(A^2))$$ is a real number. 
\end{lemma}

\begin{proof}
The characteristic polynomial $\chi_A(x)$ of $A$ is given by: $$\chi_A(x)=x^4-\tr(A) x^3 + \sigma(A) x^2 - \tr (A^{-1}) x +1.$$
Let $\lambda_1, \lambda_2, \lambda_3 $ and $\lambda_4$ are the eigenvalues of $A$. Note that $\lambda_1 \lambda_2 \lambda_3 \lambda_4=\det A=1$.  Using the Lemma \ref{wmg}, we know for each $j$, there exists $k$ such that $\lambda_j^{-1}=\overline{\lambda_k}$. So, $\tr A^{-1} = \overline{\tr}(A)$ and  
\begin{eqnarray*}
\sigma &=&\lambda_1\lambda_2+\lambda_3\lambda_4
+\lambda_1\lambda_3+\lambda_2\lambda_4
+\lambda_1\lambda_4+\lambda_2\lambda_3\\
&=&\lambda_3^{-1}\lambda_4^{-1}+\lambda_1^{-1}\lambda_2^{-1}
+\lambda_2^{-1}\lambda_4^{-1}+\lambda_1^{-1}\lambda_3^{-1}
+\lambda_2^{-1}\lambda_3^{-1}+\lambda_1^{-1}\lambda_4^{-1} \\
&=&\overline{\lambda}_1\overline{\lambda}_2+\overline{\lambda}_3\overline{\lambda}_4
+\overline{\lambda}_1\overline{\lambda}_3+\overline{\lambda}_2\overline{\lambda}_4
+\overline{\lambda}_1\overline{\lambda}_4+\overline{\lambda}_2\overline{\lambda}_3=\overline{\sigma}
\end{eqnarray*}\end{proof}

\begin{corollary}\label{realcoord}
There are 39 real polynomials coming from the real and imaginary parts of the traces in the above theorem that determine the class of any polystable pair in $\SU(3,1)$.
\end{corollary}

\begin{proof} 
From Lemma \ref{realrel} we have 5 further real relations coming from 5 pairs of invariants of the form $(\tr(\wb),\tr(\wb^2))$ in the proof of Theorem \ref{su31-traces}.\end{proof}

\subsection{Reducible pairs in \texorpdfstring{$\SU(3, 1)$}{SU(3,1)}}
The holomorphic isometry group of complex hyperbolic space $$\ch^3:=\{\z\in\C^{3,1}:\langle\z,\z \rangle<0\}/\C^*$$ is the projective unitary group $${\rm PSU}(3,1)={\rm SU }(3,1)/\{\pm \id,\pm i\id\}.$$ Based on their fixed points, holomorphic isometries of $\ch^3$ are classified as follows:
\begin{enumerate}
\item An isometry is \emph{elliptic} if it fixes at least one point of $\ch^3$.
\item An isometry is \emph{loxodromic} if it is non-elliptic and fixes exactly two points of $\partial\ch^3$, one of which is attracting and other repelling.  
\item An isometry is \emph{parabolic} if it is non-elliptic and fixes exactly one point of $\partial\ch^3$.
\end{enumerate}

In \cite{gp2}, 15 real coordinates are constructed that determine a generic pair of loxodromic elements in $\SU(3,1)$. We show in this final subsection that for {\it reducible} loxodromic representations we can determine the conjugation class with fewer coordinates.  

This is in contrast with Corollary \ref{realcoord} where 39 coordinates are shown to globally distinguish all (polystable) conjugation classes of pairs in $\SU(3,1)$. The difference between these two results is the invariant theory used to establish Corollary \ref{realcoord} significantly simplifies when applied to the reducible locus.

\begin{theorem}\label{red-lox-traces}
Let $(A,B)$ be a reducible loxodromic pair in ${\rm SU }(3,1)$. 
\begin{itemize}
\item[(1)] Let $(A, B)$ preserve a complex line in $\ch^3$. Then
 $( A,B)$ is determined up to conjugacy in ${\rm SU}(3,1)$ by 
${\rm tr}(A),$ $ {\rm tr}(B),$  $ {\rm  tr}(AB)$, $ \sigma(A)$,   $\sigma(B)$ and  $ \sigma(AB)$.
\item[(2)] Let $(A, B)$ preserve a complex plane in $\ch^3$. Then
 $( A,B)$ is determined up to conjugacy in ${\rm SU}(3,1)$ by 
${\rm tr}(A),$ $ {\rm tr}(B),$  $ {\rm  tr}(AB)$,  
$\tr( A^{-1} B)$,  $\tr([A, B])$, $\sigma(A)$ and $\sigma(B)$.
\end{itemize}
\end{theorem} 
\begin{proof}
Let $G=\langle A, B \rangle$ is a reducible polystable subgroup of ${\rm 
SU}(3,1)$ generated by loxodromic elements.  Let $\W$ be the $G$-invariant subspace of $\ch^3$. Then the following cases occur: 

\subsubsection*{Case 1.} $\W$ projects to a complex line in $\ch^3$. Then, it is a two dimensional subspace of $\C^{3,1}$ of signature $(1,1)$. Without loss of generality we consider $(A, B)$ as a pair of elements from ${\rm S}({\rm U}(1,1) \times {\rm U}(2))$. Let 
$A=\begin{pmatrix} A_1 & 0 \\ 0 & A_2 \end{pmatrix}$ and $B=\begin{pmatrix} B_1 & 0 \\ 0 & B_2 \end{pmatrix}$, where $A_1, B_1 \in {\rm U}(1,1)$ and $A_2, B_2 \in {\rm U}(2)$.  Note that $A_1$ and $B_1$ have eigenvalues with modulus different from one, and $A_2$ and $B_2$ have eigenvalues of modulus one, see \cite{gpp}.  In particular, we can determine the components from the eigenvalues.

A loxodromic element in ${\rm SU}(3,1)$ is completely determined up to conjugacy by its characteristic polynomial, that is, by $\tr$ and $\sigma$ of the element (see \cite[Lemma 4.1]{gpp}). Hence  $\tr(A)$, $\tr(B)$, $\tr(AB)$, $\sigma(A)$, $\sigma(B)$ and $\sigma(AB)$ completely determine the eigenvalues of $A$, $B$ and $AB$, and consequently they determine the traces of $A_i$, $B_i$, $A_i^2$, $B_i^2$, and $A_i B_i$, $i=1,2$. 

Applying the Fricke-Vogt Theorem (\cite[Theorem A]{gold2}) we see that $(A_i,  B_i)$, $i=1,2$, are determined up to conjugacy in their respective components by $\tr(A_i)$, $\tr(B_i)$, $\tr(A_i B_i)$, $\tr( A_i^2)$ and $\tr (B_i^2)$.  The result follows.

\subsubsection*{Case 2.} $\W$ is a two dimensional totally geodesic subspace, 
i.e. $\W$ has signature $(2,1)$ as a subspace of $\C^{3,1}$. Without loss of generality, we can assume 
$(A, B)$ is a pair in ${\rm S}({\rm U}(2,1) \times{ \rm U}(1))$. 
Let $A=\begin{pmatrix} A_0 & 0 \\ 0 & \lambda \end{pmatrix}$ and $B=\begin{pmatrix} B_0 & 0 \\ 0 & \mu \end{pmatrix}$, where $A_0$ and $B_0$ are elements of ${\rm U}(2,1)$ and $\lambda, \mu \in {\rm U}(1)$.  It follows from \cite[Theorem 4.8 ]{parker} that the ${\rm U}(2,1)$ component of the pair is determined by 
$\tr (A_0)$, $\tr (B_0)$, $\tr (A_0B_0)$, $\tr( A_0^{-1} B_0)$,  $\tr([A_0, B_0])$, $\det A_0$ and $\det B_0$. Now $\det A_0=\lambda^{-1}$ and $\det B_0=\mu^{-1}$ and these are given by eigenvalues of $A$ and $B$ respectively. Thus $\tr (A), \sigma(A)$ and $\tr(B), \sigma(B)$ determine $\lambda$ and $\mu$ completely. 
So, the pair $(A, B)$ is determined up to conjugacy  by $\tr (A)$, $\tr (B)$, $\tr (AB)$, 
$\tr( A^{-1} B)$,  $\tr([A, B])$, $\sigma(A)$ and $\sigma(B)$.\end{proof} 

\def\cdprime{$''$} \def\Dbar{\leavevmode\lower.6ex\hbox to 0pt{\hskip-.23ex
  \accent"16\hss}D}

\end{document}